\documentclass[12pt]{amsart}

\usepackage[hmargin=0.8in,height=8.6in]{geometry}
\usepackage{amssymb,amsthm, times}
\usepackage{delarray,verbatim}
\usepackage{ifpdf}
\ifpdf
\usepackage[pdftex]{graphicx}
 \else
\usepackage[dvips]{graphicx}
 \fi

\usepackage{bm}
\usepackage{multicol}

\usepackage{ifpdf}
\usepackage{color}
\definecolor{webgreen}{rgb}{0,.5,0}
\definecolor{webbrown}{rgb}{.8,0,0}
\definecolor{emphcolor}{rgb}{0.95,0.95,0.95}

\usepackage{hyperref}
\hypersetup{%
          colorlinks=true,
          linkcolor=webbrown,
          filecolor=webbrown,
          citecolor=webgreen,
          breaklinks=true}
\ifpdf \hypersetup{pdftex,
            pdfstartview=FitH, 
            bookmarksopen=true,
            bookmarksnumbered=true
} \else \hypersetup{dvips} \fi

\numberwithin{equation}{section}

\newtheorem{proposition}{Proposition}[section]

\newtheorem{remark}{Remark}[section]

\newtheorem{assump}{Assumption}[section]

\numberwithin{remark}{section} \numberwithin{proposition}{section}
\numberwithin{corollary}{section}
\newcommand {\R}{\mathbb{R}}

\newcommand{\lev}{L\'{e}vy }

\title[Cash Management and Control Band Policies for spectrally one-sided L\'{e}vy processes]{Cash Management and Control Band Policies for spectrally one-sided L\'{e}vy processes}
\thanks{This version: \today. }
\thanks{$\dagger$\,  Department of Mathematics,
Faculty of Engineering Science, Kansai University, 3-3-35 Yamate-cho, Suita-shi, Osaka 564-8680, Japan. Email: \mbox{{\em
kyamazak@kansai-u.ac.jp}}.  Tel: +81-6-6368-1527. }

\author[K. Yamazaki]{Kazutoshi Yamazaki$^\dagger$}
\date{}
\begin{document}

\begin{abstract}
We study the control band policy arising in the context of cash balance management. A policy is specified by four parameters $(d,D,U,u)$. The controller pushes the process up to $D$ as soon as it goes below $d$
and pushes down to $U$ as soon as it goes above $u$, while he does not intervene whenever it is within the set $(d, u)$.  We focus on the case when the underlying process is a spectrally one-sided L\'{e}vy process and obtain the expected fixed and proportional controlling costs as well as the holding costs under the band policy.
\end{abstract}

\maketitle
{\noindent \small{\textbf{Keywords:}\, cash balance management;  impulse control; \lev processes; scale functions}}

\section{Introduction}
In a cash balance management problem, one continuously monitors and modifies the cash balance that fluctuates stochastically over time.
  In a most general model, a controller is allowed, at a cost, to both increase and decrease the balance so as to prevent the excess and shortage. The excess and shortage costs, collectively called the \emph{holding costs}, are modeled by (typically a convex) function of the balance integrated over time. The \emph{controlling costs} consist of fixed and proportional costs, where the former is incurred at each adjustment whereas the latter is proportional to the adjustment quantity. The objective is to minimize the sum of expected values of these costs.
      
In most of the existing literature, the common goal is to show the optimality of the \emph{band policy} that is  specified by four parameters $(d,D,U,u)$ such that $d < u$ and $D,U \in (d,u)$: the controller pushes the balance up to $D$ as soon as it goes below $d$
and pushes down to $U$ as soon as it exceeds $u$; he does not intervene whenever it is within the set $[d, u]$. To our best knowledge, the existing optimality results are limited only for the Brownian motion (with a drift) case.  In particular, Constantinides and Richard \cite{constantinides1978existence}, Harrison and Taylor \cite{MR0469463}, Harrison et al.\ \cite{MR716124} solve for the linear holding cost case; Buckley and Korn \cite{buckley1998optimal} solve for  the quadratic holding cost case.

In this paper, we study the band policy of the same form by generalizing the underlying process to a class of spectrally negative L\'{e}vy processes; namely, the cash balance, in the absence of control, follows a general L\'{e}vy process with only negative jumps.  We obtain the associated net present values (NPV) of the total discounted controlling costs as well as those of the holding costs.  While it is out of scope of this paper, its potential application lies in obtaining the solution to the cash management problem by choosing appropriately the values of  $(d, D, U, u)$ and show the quasi-variational inequalities (QVI) of Bensoussan and Lions \cite{Bensoussan_Lions_1984}.

While the inclusion of jumps makes the problem significantly harder, there have recently been several results on related stochastic control problems.  In particular, there are two special cases of the cash balance management problem that have been solved analytically for a general spectrally negative L\'{e}vy process.  First, under the additional constraint that the process can only be augmented, a two-parameter band policy, known as the $(s,S)$-policy, has been shown to be optimal by Yamazaki \cite{Yamazaki_2013} (as a generalization of the previous results by \cite{Bensoussan_2009, Bensoussan_2005} for processes with compound Poisson jumps).  Second, in the absence of fixed controlling costs, Baurdoux and Yamazaki \cite{baurdoux_yamazaki_2014} show the optimality of another two-parameter band policy where the optimally controlled process becomes a doubly reflected L\'{e}vy process of  \cite{Avram_2007, Pistorius_2003}.  For other stochastic control problems where the optimal policy is characterized by two parameters, we refer the reader to \cite{Bayraktar_2013,Loeffen_2009_2} for optimal dividend problems with fixed transaction costs and \cite{Leung_Yamazaki_2011, Hernandez_Yamazaki_2013} for two-player stochastic games.

The objective of this paper is to obtain semi-analytical expressions of the NPV's of the total discounted costs associated with the band policy.  Following the same paths of the above mentioned papers, we use the scale function to  efficiently write these quantities.  We expect these expressions to be beneficial in solving the cash management problem; the forms written in terms of the scale function can potentially help one to analyze the smoothness of the value function and to verify the optimality of a candidate band policy.

The rest of the paper is organized as follows.  Section \ref{section_model} reviews the spectrally negative L\'{e}vy process, the band policy, and the scale function.  Sections \ref{section_controlling_costs}
 and \ref{section_holding_costs} obtain, using the scale function, the NPV's of the controlling and holding costs, respectively. Section \ref{section_conclusion} concludes the paper with discussions on its contributions as well as potential challenges in its application in cash management problems.

\section{Mathematical Formulation} \label{section_model}

Let $(\Omega, \mathcal{F}, \mathbb{P})$ be a probability space hosting a \emph{spectrally negative L\'{e}vy process} $X = \left\{X_t; t \geq 0 \right\}$ whose \emph{Laplace exponent} is given by
\begin{align}
\psi(s)  := \log \mathbb{E} \left[ e^{s X_1} \right] =  c s +\frac{1}{2}\sigma^2 s^2 + \int_{(-\infty,0)} (e^{s z}-1 - s z 1_{\{-1 < z < 0\}}) \nu ({\rm d} z), \quad s \geq 0, \label{laplace_spectrally_positive}
\end{align}
where $\nu$ is a L\'{e}vy measure with the support $(-\infty,0)$ that satisfies the integrability condition $\int_{(-\infty,0)} (1 \wedge z^2) \nu({\rm d} z) < \infty$.  It has paths of bounded variation if and only if $\sigma = 0$ and $\int_{(-1,0)}|z|\, \nu({\rm d} z) < \infty$; in this case, we write \eqref{laplace_spectrally_positive} as
\begin{align*}
\psi(s)   =  \delta s +\int_{(-\infty,0)} (e^{s z}-1 ) \nu ({\rm d} z), \quad s \geq 0,
\end{align*}
with $\delta := c - \int_{(-1,0)}z\, \nu({\rm d} z)$.  We exclude the case in which $X$ is the negative of a subordinator (i.e., $X$ has monotone paths a.s.). This assumption implies that $\delta > 0$ when $X$ is of bounded variation. 
Let $\mathbb{P}_x$ be the conditional probability under which $X_0 = x$ (also let $\mathbb{P} \equiv \mathbb{P}_0$), and let $\mathbb{F} := \left\{ \mathcal{F}_t: t \geq 0 \right\}$ be the filtration generated by $X$.  

Fix $(d,D,U,u)$ such that $d < u$ and $D,U \in (d,u)$.   We consider adjusting the process $X$ by adding and subtracting the processes $R \equiv R(d, D,U,u)$ and $L \equiv L(d,D,U,u)$, respectively;  the resulting controlled process becomes:
\begin{align*}
A_t = A_{t}(d, D, U, u) := X_t  + R_t - L_t, \quad t \geq 0.
\end{align*}
The process $R$ pushes the process up to $D$ as soon as it goes below $d$
while the process $L$ pushes it down to $U$ as soon as it goes above $u$.  We consider the right-continuous versions for $R$ and $L$.  For the sake of completeness, we construct the processes as follows.  In doing so, we also define an auxiliary process
\begin{align*}
\tilde{A}_t := A_{t-} + \Delta X_t, = A_t - (\Delta R_t - \Delta L_t), \quad t \geq 0,
\end{align*}
which can be understood as the \emph{pre-controlled} process that does not reflect at $t$ the adjustments made by the processes $R_t$ and $L_t$.  Here and throughout, let $\Delta \xi_t := \xi_t - \xi_{t-}$, for any right-continuous process $\xi$.
%
%
%
%

\begin{center}
\line(1,0){300}
\end{center}

 \textbf{Construction of the processes $A$, $\tilde{A}$, $L$ and $R$}
\begin{description}
\item[Step 1] Set $A_{0-}=\tilde{A}_{0} = x$ and $L_{0-} = R_{0-} = 0$.
\begin{description}
\item[Step 1-1] 
If $d \leq x \leq u$, set 
\begin{align*}
A_0  = x \quad \textrm{and} \quad L_0 = R_0 = 0. 
\end{align*}
If $x < d$, set 
\begin{align*}
A_0 =  D, \quad L_0 =  0, \quad  \textrm{and} \quad R_0 = D-x. 
\end{align*}
If $x > u$, set 
\begin{align*}
A_0 =  U, \quad L_0 =  x-U, \quad  \textrm{and} \quad R_0 = 0. 
\end{align*}
\item[Step 1-2] Set $n = 0$ and define $T^{(0)} = 0$.
\end{description}
\item[Step 2]  
\begin{description}
\item[Step 2-1]Set
\begin{align*}
\tilde{A}_t = A_{T^{(n)}} + (X_t - X_{T^{(n)}}), \quad T^{(n)} < t \leq T^{(n+1)} := T_u^{(n+1)+} \wedge T_d^{(n+1)-}
\end{align*}
where we define \begin{align*}
T_u^{(n+1)+} &:= \inf \left\{ t \geq T^{(n)}: \tilde{A}_t >  u \right\}, \\  T_d^{(n+1)-} &:= \inf \left\{ t \geq T^{(n)}: \tilde{A}_t < d \right\}.
\end{align*}
\item[Step 2-2]Set $A_t = \tilde{A}_t$,  $R_t = R_{T^{(n)}}$ and $L_t = L_{T^{(n)}}$ for $T^{(n)} < t < T^{(n+1)}$ and
\begin{align*}
A_{T^{(n+1)}} &= \left\{ \begin{array}{ll} U, & \textrm{ if } \; T^{(n+1)} = T_u^{(n+1)+}, \\ D, & \textrm{ if } \; T^{(n+1)} = T_d^{(n+1)-}, \end{array} \right. \\
R_{T^{(n+1)}} &= \left\{ \begin{array}{ll} R_{T^{(n)}}, & \textrm{ if } \; T^{(n+1)} = T_u^{(n+1)+}, \\ R_{T^{(n)}} + (D-\tilde{A}_{T^{(n+1)}}), & \textrm{ if } \; T^{(n+1)} = T_d^{(n+1)-}, \end{array} \right. \\
L_{T^{(n+1)}} &= \left\{ \begin{array}{ll} L_{T^{(n)}} + (u-U), & \textrm{ if } \; T^{(n+1)} = T_u^{(n+1)+}, \\ L_{T^{(n)}}, & \textrm{ if } \; T^{(n+1)} = T_d^{(n+1)-}. \end{array} \right.
\end{align*}
\item[Step 2-3] Increment the value of $n$ by $1$ and go back to \textbf{Step 2-1}.
\end{description}
\end{description}
\begin{center}
\line(1,0){300}
\end{center}

In the algorithm above, the processes are first initialized in \textbf{Step 1}.  In the constructions in \textbf{Step 2},
the process $R_t$ (resp.\ $L_t$) stays constant while the pre-controlled process $\tilde{A}_t$ remains on $[d, \infty)$ (resp.\ $(-\infty, u]$), and it increases by $D- \tilde{A}_{t}$ (resp.\ $\tilde{A}_t - U$) as soon as $\tilde{A}$ enters $(-\infty, d)$ (resp.\ $(u, \infty)$).  By construction, $R$ and $L$ are non-decreasing a.s.\ and the controlled process $A_{t}$ always remains on the interval $[d,u]$.  It is easy to see that these processes are $\mathbb{F}$-adapted; in particular, the processes $A$ and $\tilde{A}$ are strong Markov processes.

\subsection{Scale functions}  We conclude this section with a brief review on the scale function.

Fix $q > 0$. For any spectrally negative L\'{e}vy process, there exists a function called  the  $q$-scale function 
\begin{align*}
W^{(q)}: \R \rightarrow [0,\infty), 
\end{align*}
which is zero on $(-\infty,0)$, continuous and strictly increasing on $[0,\infty)$, and is characterized by the Laplace transform:
\begin{align*}
\int_0^\infty e^{-s x} W^{(q)}(x) {\rm d} x = \frac 1
{\psi(s)-q}, \qquad s > {\Phi(q)},
\end{align*}
where
\begin{equation}
{\Phi(q)} :=\sup\{\lambda \geq 0: \psi(\lambda)=q\}. \notag
\end{equation}
Here, the Laplace exponent $\psi$ in \eqref{laplace_spectrally_positive} is known to be zero at the origin and convex on $[0,\infty)$; therefore ${\Phi(q)}$ is well defined and is strictly positive as $q > 0$.   We also define, for $x \in \R$,
\begin{align*}
\overline{W}^{(q)}(x) &:=  \int_0^x W^{(q)}(y) {\rm d} y, \\
Z^{(q)}(x) &:= 1 + q \overline{W}^{(q)}(x),  \\
\overline{Z}^{(q)}(x) &:= \int_0^x Z^{(q)} (z) {\rm d} z = x + q \int_0^x \int_0^z W^{(q)} (w) {\rm d} w {\rm d} z.
\end{align*}
Because $W^{(q)}$ is uniformly zero on the negative half line, we have $Z^{(q)}(x) = 1$ and $\overline{Z}^{(q)}(x) = x$ for $x \leq 0$.

Let us define the first down- and up-crossing times, respectively, of $X$ by
\begin{align}
\label{first_passage_time}
\tau_b^- := \inf \left\{ t \geq 0: X_t < b \right\} \quad \textrm{and} \quad \tau_b^+ := \inf \left\{ t \geq 0: X_t >  b \right\}, \quad b \in \R.
\end{align}
Then, for any $b > 0$ and $x \leq b$,
\begin{align}
\label{two_sided_exit}
\mathbb{E}_x \left[ e^{-q \tau_b^+} 1_{\left\{ \tau_b^+ < \tau_0^- \right\}}\right] = \frac {W^{(q)}(x)}  {W^{(q)}(b)} \quad \textrm{and} \quad
 \mathbb{E}_x \left[ e^{-q \tau_0^-} 1_{\left\{ \tau_b^+ > \tau_0^- \right\}}\right] = Z^{(q)}(x) -  Z^{(q)}(b) \frac {W^{(q)}(x)}  {W^{(q)}(b)}.
\end{align}
In addition, as in Theorem 8.7 of \cite{Kyprianou_2006}, for any measurable function $f$ bounded on $[d,u]$, we have
\begin{align}
\mathbb{E}_x \left[ \int_0^{\tau_{d}^- \wedge \tau^+_u} e^{-qt} f(X_t){\rm d} t\right] &= \varphi_d(u;f) \frac {W^{(q)}(x-d)}  {W^{(q)}(u-d)}  - \varphi_d(x;f),  \label{resolvent_formula}
\end{align}
where
\begin{align} \label{def_psi_phi}
\varphi_{d'} (x';f) &:= \int_{d'}^{x'} W^{(q)} (x'-y) f(y) {\rm d} y, \quad d',x' \in \R.
\end{align}


\begin{remark} \label{remark_smoothness_zero}
\begin{enumerate}
\item If $X$ is of unbounded variation or the L\'{e}vy measure is atomless, it is known that $W^{(q)}$ is $C^1(\R \backslash \{0\})$; see, e.g.,\ \cite{Chan_2009}.  Hence, 
\begin{enumerate}
\item $Z^{(q)}$ is $C^1 (\R \backslash \{0\})$ and $C^0 (\R)$ for the bounded variation case, while it is $C^2(\R \backslash \{0\})$ and $C^1 (\R)$ for the unbounded variation case, and
\item $\overline{Z}^{(q)}$ is $C^2(\R \backslash \{0\})$ and $C^1 (\R)$ for the bounded variation case, while it is $C^3(\R \backslash \{0\})$ and $C^2 (\R)$ for the unbounded variation case.
\end{enumerate}
\item Regarding the asymptotic behavior near zero, as in Lemmas 4.3 and 4.4 of \cite{Kyprianou_Surya_2007}, 
\begin{align}\label{eq:Wq0}
W^{(q)} (0) &= \left\{ \begin{array}{ll} 0, & \textrm{if $X$ is of unbounded
variation,} \\ \frac 1 {\delta}, & \textrm{if $X$ is of bounded variation,}
\end{array} \right. \\
\label{eq:Wqp0}
W^{(q)'} (0+) &:= \lim_{x \downarrow 0}W^{(q)'} (x) =
\left\{ \begin{array}{ll}  \frac 2 {\sigma^2}, & \textrm{if }\sigma > 0, \\
\infty, & \textrm{if }\sigma = 0 \; \textrm{and} \; \nu(-\infty,0) = \infty, \\
\frac {q + \nu(-\infty, 0)} {\delta^2}, &  \textrm{if }\sigma = 0 \; \textrm{and} \; \nu(-\infty, 0) < \infty.
\end{array} \right.
\end{align}
\item As in (8.18) and Lemma 8.2 of \cite{Kyprianou_2006},
\begin{align*}
\frac {W^{(q)'}(y+)} {W^{(q)}(y)} \leq \frac {W^{(q)'}(x+)} {W^{(q)}(x)},  \quad  y > x > 0.
\end{align*}
In all cases, $W^{(q)'}(x-) \geq W^{(q)'}(x+)$ for all $x >0$.
\end{enumerate}
\end{remark}

\section{Controlling costs} \label{section_controlling_costs}

In this section, we compute the controlling costs given by, for all $x \in \R$, 
\begin{align}
v_L(x) &:= \mathbb{E}_x  \Big[ \sum_{0 \leq s \leq t} e^{-qs} (\Delta L_s + k_L) 1_{\{ \Delta L_s > 0\}}  \Big],  \label{v_L}\\
v_R(x) &:= \mathbb{E}_x  \Big[ \sum_{0 \leq s \leq t} e^{-qs}  (\Delta R_s  + k_R)1_{\{ \Delta R_s > 0\}}  \Big], \label{v_R}
\end{align}
for given constants $k_L, k_R \in \R$.   Throughout, we fix $(d,D,U,u)$ such that $d < u$ and $D,U \in (d,u)$.

We shall write these in terms of the scale function as reviewed above.
Because both $W^{(q)}$ and $\overline{W}^{(q)}$ are nondecreasing, we can define the measures $W^{(q)} ({\rm d} x)$ and $\overline{W}^{(q)} ({\rm d} x)$ such that, for any $y > x > 0$, 
\begin{align*}
W^{(q)}(x,y) = W^{(q)} (y) - W^{(q)} (x) \quad \textrm{and} \quad \overline{W}^{(q)}(x,y) = \overline{W}^{(q)} (y) - \overline{W}^{(q)} (x).
\end{align*}
 Let us also define
\begin{align*}
\Xi(d, D, U,u)  := \overline{W}^{(q)}(U-d, u-d)   W^{(q)}(D-d) - W^{(q)}(U-d,u-d)  \overline{W}^{(q)}(D-d).
\end{align*}

We first obtain the expression for \eqref{v_L}.
\begin{proposition} 
 Let 
 \begin{align*}
 \epsilon_L := (u-U) + k_L.
 \end{align*}
 \begin{enumerate}
 \item  For all $d \leq x \leq u$,
\begin{align*}
&v_L(x) = \frac {\epsilon_L} {\Xi(d,D, U,u)} \left[   Z^{(q)}(x-d) \frac {W^{(q)}(D-d)} q - W^{(q)}(x-d)   \overline{W}^{(q)}(D-d) \right].
\end{align*}

\item For all $x > u$, 
\begin{align*}
v_L(x) &= (x-U) + k_L + v_L(U) \\ &=  (x-U) + k_L +
\frac  {\epsilon_L} q \frac {W^{(q)}(D-d)  {Z^{(q)}(U-d)}  - q {W^{(q)}(U-d)}  \overline{W}^{(q)}(D-d)} {\Xi(d,D, U,u)}.
\end{align*}
\item  For all $x <d$,
\begin{align*}
v_L(x) = v_L(D) =\frac {\epsilon_L} q \frac {W^{(q)}(D-d)} {\Xi(d,D, U,u)}.
\end{align*}
\end{enumerate}
\end{proposition}
\begin{proof} 
Fix $d \leq x \leq u$. Suppose 
\begin{align}
T_b^+ := \inf \left\{ t \geq 0: \tilde{A}_t >  b \right\} \quad \textrm{and} \quad  T_b^- := \inf \left\{ t \geq 0: \tilde{A}_t < b \right\} , \quad b \in \R. \label{hitting_time_A_tilde}
\end{align}
Because the law of $\{ \tilde{A}_t;  t \leq T_u^+ \wedge T_d^- \}$ and that of $\{ X_t; t \leq  \tau_u^+ \wedge \tau_d^- \}$ are the same (see the above construction of the process $\tilde{A}$),  the strong Markov property  and $\tilde{A}_{T_u^+}   = u$ on $\{T_u^+ < \infty \}$ (due to the fact that $X$ has no positive jumps) gives
\begin{align*}
v_L(x) &= \mathbb{E}_x \left[e^{-q T_u^+} 1_{\{T_u^+ < T_d^- \}}\right] (v_L(U)+\epsilon_L) +  \mathbb{E}_x \left[e^{-q T_d^-} 1_{\{T_u^+ > T_d^- \}}\right] v_L(D) \\&= \mathbb{E}_x \left[e^{-q \tau_u^+} 1_{\{\tau_u^+ < \tau_d^- \}}\right] (v_L(U)+\epsilon_L) +  \mathbb{E}_x \left[e^{-q \tau_d^-} 1_{\{\tau_u^+ > \tau_d^- \}}\right] v_L(D).
\end{align*}
Hence, by \eqref{two_sided_exit}, 
\begin{align} \label{v_split}
\begin{split}
v_L(x) &= \frac {W^{(q)}(x-d)}  {W^{(q)}(u-d)} [v_L(U)+\epsilon_L] +  \left[ Z^{(q)}(x-d) -  Z^{(q)}(u-d) \frac {W^{(q)}(x-d)}  {W^{(q)}(u-d)} \right] v_L(D)  \\
&= \frac {W^{(q)}(x-d)}  {W^{(q)}(u-d)} \left[ v_L(U) +\epsilon_L- Z^{(q)}(u-d)v_L(D)  \right] 
+  Z^{(q)}(x-d) v_L(D).
\end{split}
\end{align}
In particular, by substituting $x=U, D$, we obtain
\begin{align}
v_L(U) 
&= \frac {W^{(q)}(U-d)}  {W^{(q)}(u-d)} \left[ v_L(U) +\epsilon_L- Z^{(q)}(u-d)v_L(D)  \right] 
+  Z^{(q)}(U-d) v_L(D), \nonumber \\
v_L(D) 
&= \frac {W^{(q)}(D-d)}  {W^{(q)}(u-d)} \left[ v_L(U) +\epsilon_L- Z^{(q)}(u-d)v_L(D)  \right] 
+  Z^{(q)}(D-d) v_L(D). \label{v_L_quadratic_eqn}
\end{align}
By computing $v_L(U) - v_L(D) W^{(q)}(U-d) / W^{(q)}(D-d)$, we attain the relation:
\begin{align}
v_L(U)  = v_L(D) \left[ Z^{(q)}(U-d) - \frac {W^{(q)}(U-d)}  {W^{(q)}(D-d)}  q \overline{W}^{(q)}(D-d)  \right]. \label{relation_v_L_U}
\end{align}
Substituting this back in \eqref{v_L_quadratic_eqn} and solving for $v_L(D)$, we obtain
\begin{align}
v_L(D) 
&= \frac {\epsilon_L} q \frac {W^{(q)}(D-d)} {\Xi(d,D, U,u)}. \label{v_L_D} 
\end{align}
In addition, substituting this in \eqref{relation_v_L_U} gives
\begin{align}
v_L(U)  &= 
\frac  {\epsilon_L} q \frac {W^{(q)}(D-d)  {Z^{(q)}(U-d)}  - q {W^{(q)}(U-d)}  \overline{W}^{(q)}(D-d)} {\Xi(d,D, U,u)}. \label{v_L_U} 
\end{align}
These together with \eqref{v_split} complete the proof of (1).
The proofs of (2) and (3) are immediate by the construction of the process $L$ and by \eqref{v_L_D} and \eqref{v_L_U}.
\end{proof}

We now move on to obtaining the expression for \eqref{v_R}.  Toward this end, we assume that the first moment of $X_t$ is finite.
\begin{assump} \label{assump_mu} Suppose  $\mu :=\mathbb{E} [X_1]  = \psi'(0+) \in (-\infty,\infty)$.
\end{assump}
We define the following short-hand notations:
\begin{align*}
\epsilon_R &:= (D-d) + k_R, \\
Y^{(q)}(y) &\equiv Y^{(q)}(y; \epsilon_R):=  \overline{Z}^{(q)}(y) + \frac {\mu} q - \Big(\frac {\mu} q +\epsilon_R \Big)  Z^{(q)}(y), \quad y \in \R.
\end{align*}
\begin{proposition}  Suppose Assumption  \ref{assump_mu} holds.
\begin{enumerate}
\item For $d \leq x \leq u$,
\begin{align*}
&{v_R(x)}   \\
&=    \frac {Z^{(q)}(x-d)} q   \frac {   W^{(q)}(D-d)  \big[Y^{(q)}(u-d) - Y^{(q)}(U-d) \big] -   {Y^{(q)}(D-d)  W^{(q)}(U-d,u-d)}   } {\Xi(d, D, U,u)}\\ & - Y^{(q)}(x-d) \\ &+    W^{(q)}(x-d) \frac {  Y^{(q)}(D-d)   \overline{W}^{(q)}(U-d,u-d) - \overline{W}^{(q)}(D-d)  \big[     Y^{(q)}(u-d) - Y^{(q)}(U-d) \big] } {\Xi(d, D, U,u)}. 
\end{align*}
\item For all $x > u$,
\begin{align*}
v_R(x) &= v_R(U) \\ &=  
 \Big[ Z^{(q)}(U-d) -  \frac {W^{(q)}(U-d)} {W^{(q)}(D-d)}   q \overline{W}^{(q)}(D-d) \Big]  \\ &\times \frac {[Y^{(q)}(u-d) - Y^{(q)}(U-d)]  W^{(q)}(D-d)  -{W^{(q)}(U-d,u-d)}    Y^{(q)}(D-d)} {q\Xi(d, D, U,u)}
\\
 &+ \frac {W^{(q)}(U-d)} {W^{(q)}(D-d)}  Y^{(q)}(D-d) - Y^{(q)}(U-d).
\end{align*}
\item  For all $x <d$, 
\begin{align*}
v_R(x) &= (D-x) + k_R + v_R(D) \\ &=  (D-x) + k_R \\ &+  \frac {[Y^{(q)}(u-d) - Y^{(q)}(U-d)]  W^{(q)}(D-d)  -{W^{(q)}(U-d,u-d)}    Y^{(q)}(D-d)} {q\Xi(d, D, U,u)}.
\end{align*}

\end{enumerate}

\end{proposition}
\begin{proof}
Fix $d \leq x \leq u$.
Because the law of $\{ \tilde{A}_t;  t \leq T_u^+ \wedge T_d^- \}$ and that of $\{ X_t; t \leq  \tau_u^+ \wedge \tau_d^- \}$ are the same (see \eqref{hitting_time_A_tilde}), the strong Markov property gives
\begin{align*}
v_R(x) &= \mathbb{E}_x \left[e^{-q T_u^+} 1_{\{T_u^+ < T_d^- \}}\right] v_R(U) +  \mathbb{E}_x \left[e^{-q T_d^-} 1_{\{T_u^+ > T_d^- \}}\right] v_R(D)  \\
&+  \mathbb{E}_x \left[ e^{-q T_{d}^-} 1_{\{ T_{d}^- < T_u^+\}} (d-\tilde{A}_{T_{d}^-} + \epsilon_R) \right]
\\&=  \mathbb{E}_x \left[e^{-q \tau_u^+} 1_{\{\tau_u^+ < \tau_d^- \}}\right] v_R(U) +  \mathbb{E}_x \left[e^{-q \tau_d^-} 1_{\{\tau_u^+ > \tau_d^- \}}\right] v_R(D)  \\
&+  \mathbb{E}_x \left[ e^{-q \tau_{d}^-} 1_{\{ \tau_{d}^- < \tau_u^+\}} (d-X_{\tau_{d}^-} + \epsilon_R) \right].
\end{align*}
Here, Lemma 3.1 of  \cite{Bayraktar_2013}
and \eqref{two_sided_exit} give
\begin{align*}
\mathbb{E}_x \left[ e^{-q \tau_{d}^-} 1_{\{ \tau_{d}^- < \tau_u^+\}} (d-X_{\tau_{d}^-} + \epsilon_R) \right]
&= -Y^{(q)}(x-d)  +  Y^{(q)}(u-d)  \frac {W^{(q)}(x-d)} {W^{(q)}(u-d)}.
\end{align*}

Substituting this and  \eqref{two_sided_exit},
\begin{align} \label{u_decomposition}
\begin{split}
v_R(x) &= \frac {W^{(q)}(x-d)}  {W^{(q)}(u-d)}  v_R(U) +  \left[ Z^{(q)}(x-d) -  Z^{(q)}(u-d) \frac {W^{(q)}(x-d)}  {W^{(q)}(u-d)} \right]  v_R(D) \\
&-  Y^{(q)}(x-d)  +   Y^{(q)}(u-d)  \frac {W^{(q)}(x-d)} {W^{(q)}(u-d)} \\
&=  \left[ Z^{(q)}(x-d) -  Z^{(q)}(u-d) \frac {W^{(q)}(x-d)}  {W^{(q)}(u-d)} \right]  v_R(D) - Y^{(q)}(x-d)  \\ &+  \left[  Y^{(q)}(u-d) + v_R(U)  \right] \frac {W^{(q)}(x-d)} {W^{(q)}(u-d)}.
\end{split}
\end{align}
In particular,  by setting $x = D, U$, we obtain
\begin{align}
\begin{split}
v_R(D) 
&=  \left[ Z^{(q)}(D-d) -  Z^{(q)}(u-d) \frac {W^{(q)}(D-d)}  {W^{(q)}(u-d)} \right]  v_R(D) - Y^{(q)}(D-d)  \\ &+  \left[  Y^{(q)}(u-d) + v_R(U)  \right] \frac {W^{(q)}(D-d)} {W^{(q)}(u-d)}, 
\end{split} \label{v_R_D_intermidiate}
\\
v_R(U) 
&=  \left[ Z^{(q)}(U-d) -  Z^{(q)}(u-d) \frac {W^{(q)}(U-d)}  {W^{(q)}(u-d)} \right]  v_R(D) - Y^{(q)}(U-d)  \nonumber \\ &+  \left[  Y^{(q)}(u-d) + v_R(U)  \right] \frac {W^{(q)}(U-d)} {W^{(q)}(u-d)}. \nonumber
\end{align}
In order to solve this system of equations, we compute
\begin{align*}
v_R(U) - v_R(D)  \frac {W^{(q)}(U-d)} {W^{(q)}(D-d)} &= \left[ Z^{(q)}(U-d) -  Z^{(q)}(u-d) \frac {W^{(q)}(U-d)}  {W^{(q)}(u-d)} \right]  v_R(D) \\
&-  \frac {W^{(q)}(U-d)} {W^{(q)}(D-d)}  \left[ Z^{(q)}(D-d) -  Z^{(q)}(u-d) \frac {W^{(q)}(D-d)}  {W^{(q)}(u-d)} \right]  v_R(D) \\
&+   \frac {W^{(q)}(U-d)} {W^{(q)}(D-d)} Y^{(q)}(D-d) - Y^{(q)}(U-d)   \\
&=   \Big[ Z^{(q)}(U-d)  -  \frac {W^{(q)}(U-d)} {W^{(q)}(D-d)}   Z^{(q)}(D-d)  \Big] v_R(D) \\
&+ \frac {W^{(q)}(U-d)} {W^{(q)}(D-d)}  Y^{(q)}(D-d) -Y^{(q)}(U-d),
\end{align*}
and therefore
\begin{align}
\begin{split}
v_R(U) 
&=  \Big[ Z^{(q)}(U-d) -  \frac {W^{(q)}(U-d)} {W^{(q)}(D-d)}   q \overline{W}^{(q)}(D-d) \Big]  v_R(D) \\ &+  \frac {W^{(q)}(U-d)} {W^{(q)}(D-d)}  Y^{(q)}(D-d) - Y^{(q)}(U-d). \end{split} \label{v_R_U_temp}
\end{align}
Substituting this in \eqref{v_R_D_intermidiate} and solving for $v_R(D)$ gives
\begin{align}
v_R(D)  =  \frac {[Y^{(q)}(u-d) - Y^{(q)}(U-d)]  W^{(q)}(D-d)  -{W^{(q)}(U-d,u-d)}    Y^{(q)}(D-d)} {q\Xi(d, D, U,u)}, \label{v_R_D}
\end{align}
and hence
\begin{align}\label{v_R_U}
\begin{split}
v_R(U) 
&=  \Big[ Z^{(q)}(U-d) -  \frac {W^{(q)}(U-d)} {W^{(q)}(D-d)}   q \overline{W}^{(q)}(D-d) \Big]  \\ &\times \frac {[Y^{(q)}(u-d) - Y^{(q)}(U-d)]  W^{(q)}(D-d)  -{W^{(q)}(U-d,u-d)}    Y^{(q)}(D-d)} {q\Xi(d, D, U,u)}
\\
 &+ \frac {W^{(q)}(U-d)} {W^{(q)}(D-d)}  Y^{(q)}(D-d) - Y^{(q)}(U-d).
 \end{split}
\end{align}

By \eqref{u_decomposition} and \eqref{v_R_U_temp},
\begin{align}
v_R(x) 
&=   Z^{(q)}(x-d)   v_R(D) -Y^{(q)}(x-d)+   B(d,D,U,u) \frac {W^{(q)}(x-d)} {W^{(q)}(u-d)}, \label{v_R_semifinal}
\end{align}
where \begin{align*}
 B(d,D,U,u)
&:= - q\left[ \overline{W}^{(q)}(U-d,u-d) +  \frac {W^{(q)}(U-d)} {W^{(q)}(D-d)}   \overline{W}^{(q)}(D-d) \right] v_R(D)  \\&+   \frac {W^{(q)}(U-d)} {W^{(q)}(D-d)}  Y^{(q)}(D-d) +  Y^{(q)}(u-d) - Y^{(q)}(U-d). 
\end{align*}
Here in particular
\begin{align*}
& \left[ \overline{W}^{(q)}(U-d,u-d) +  \frac {W^{(q)}(U-d)} {W^{(q)}(D-d)}   \overline{W}^{(q)}(D-d) \right] v_R(D)  \\
 &= \Big[ \frac {W^{(q)}(u-d)}  {W^{(q)}(D-d)}  \overline{W}^{(q)}(D-d)  + \frac {\Xi(d,D, U,u)} {W^{(q)}(D-d)} \Big] v_R(D)   \\
&=   \frac {W^{(q)}(u-d)}  {W^{(q)}(D-d)}  \overline{W}^{(q)}(D-d)  v_R(D) \\ &+  \frac 1 q \Big[ Y^{(q)}(u-d) - Y^{(q)}(U-d)   - \frac {W^{(q)}(U-d,u-d)} {W^{(q)}(D-d)}   Y^{(q)}(D-d)  \Big].
\end{align*}
Hence,
\begin{align*}
& B(d,D,U,u)  
=    \frac {W^{(q)}(u-d)}  {W^{(q)}(D-d)}  \left[Y^{(q)}(D-d) -q \overline{W}^{(q)}(D-d)  v_R(D)    \right] \\
&=    W^{(q)}(u-d)\frac {   Y^{(q)}(D-d)  \overline{W}^{(q)}(U-d,u-d) -  \overline{W}^{(q)}(D-d)  [Y^{(q)}(u-d) - Y^{(q)}(U-d)]} {\Xi(d, D, U,u)}.
\end{align*}
Substituting this and \eqref{v_R_D} in \eqref{v_R_semifinal}, the proof of (1) is complete.  The proofs of (2) and (3) are immediate by the construction of the process $R$ and by \eqref{v_R_D} and \eqref{v_R_U}.

\end{proof}


\section{Holding costs} \label{section_holding_costs}

Fix $(d,D,U,u)$ such that $d < u$ and $D,U \in (d,u)$ and define
\begin{align*}
w(x) = w(x;f) := \mathbb{E}_x \left[ \int_0^\infty e^{-qt} f(A_t) {\rm d} t \right],
\end{align*}
for any measurable function $f$ bounded on $[d,u]$.  We define
\begin{align*}
\Theta(d,D,U,u;f) := {W^{(q)}(D-d)}   [\varphi_d (u;f) - \varphi_d (U;f)]
- {W^{(q)}(U-d,u-d)}  \varphi_d (D;f). 
\end{align*}

\begin{proposition}
\begin{enumerate}
\item For any $d \leq x \leq u$,
\begin{align*}
w(x) &= \frac {W^{(q)}(x-d)} {W^{(q)}(D-d)}  \varphi_d (D;f)  - \varphi_d (x;f) \\ &+ \left[ \frac {Z^{(q)}(x-d)} q - {W^{(q)}(x-d)}  \frac { \overline{W}^{(q)}(D-d)} {W^{(q)}(D-d)}   \right] \frac {\Theta(d,D,U,u;f)} 
{ \Xi(d,D, U,u)}.
\end{align*}
\item For $x < d$, 
\begin{align*}
w(x) = w(D) = \frac {\Theta(d,D,U,u;f)} 
{q \Xi(d,D,U,u)}.
\end{align*}
\item For $x > u$, 
\begin{align*}
w(x) &= w(U) \\
&= \frac {\Theta(d,D,U,u;f)} 
{q \Xi(d,D,U,u)}
 \left[ Z^{(q)}(U-d)  -  q \frac {W^{(q)}(U-d)}  {W^{(q)}(D-d)} \overline{W}^{(q)}(D-d) \right] \\ &- \varphi_d (U;f)  +  \frac {W^{(q)}(U-d)}  {W^{(q)}(D-d)} \varphi_d (D;f).
\end{align*}
\end{enumerate}

\end{proposition}
\begin{proof}
Fix $d \leq x \leq u$. Again, because the law of $\{ \tilde{A}_t;  t \leq T_u^+ \wedge T_d^- \}$ and that of $\{ X_t; t \leq \tau_u^+ \wedge \tau_d^- \}$ are the same,  the strong Markov property gives
\begin{align*}
w(x) = \mathbb{E}_x \left[e^{-q \tau_u^+} 1_{\{\tau_u^+ < \tau_d^- \}}\right] w(U) +  \mathbb{E}_x \left[e^{-q \tau_d^-} 1_{\{\tau_u^+ > \tau_d^- \}}\right] w(D) +\mathbb{E}_x \left[ \int_0^{\tau_d^- \wedge \tau_u^+} e^{-qt} f(X_t) {\rm d} t \right].
\end{align*}
By \eqref{two_sided_exit} and \eqref{resolvent_formula},
\begin{align} \label{w_intermediate}
\begin{split}
w(x)
&= \frac {W^{(q)}(x-d)}  {W^{(q)}(u-d)} \left[ w(U) - Z^{(q)}(u-d)w(D) + \varphi_d (u;f)  \right] 
+  Z^{(q)}(x-d) w(D) - \varphi_d (x;f).
\end{split}
\end{align}
In particular, by setting $x=U,D$,
\begin{align}
w(U) 
&= \frac {W^{(q)}(U-d)}  {W^{(q)}(u-d)} \left[ w(U) - Z^{(q)}(u-d)w(D) + \varphi_d (u;f)  \right] 
+  Z^{(q)}(U-d) w(D) - \varphi_d (U;f), \nonumber \\
w(D) 
&= \frac {W^{(q)}(D-d)}  {W^{(q)}(u-d)} \left[ w(U) - Z^{(q)}(u-d)w(D) + \varphi_d (u;f)  \right] 
+  Z^{(q)}(D-d) w(D) - \varphi_d (D;f). \label{w_D_quadratic_eqn}
\end{align}
Hence by computing $w(U) - w(D) {W^{(q)}(U-d)} / {W^{(q)}(D-d)} $, we obtain
\begin{align*}
w(U) 
&= w(D)  \left[ Z^{(q)}(U-d)  -  q \frac {W^{(q)}(U-d)}  {W^{(q)}(D-d)} \overline{W}^{(q)}(D-d) \right] \\ &- \varphi_d (U;f)  +  \frac {W^{(q)}(U-d)}  {W^{(q)}(D-d)} \varphi_d (D;f).
\end{align*}
Substituting this in \eqref{w_D_quadratic_eqn},
\begin{align*}
w(D) 
&=   \frac {q w(D)} {W^{(q)}(u-d)}  \left[ -{W^{(q)}(D-d)} \overline{W}^{(q)}(U-d,u-d)-   {W^{(q)}(U-d)}   \overline{W}^{(q)}(D-d) \right]  \\ &+    \frac {W^{(q)}(U-d)}  {W^{(q)}(u-d)} \varphi_d (D;f)  + \frac {W^{(q)}(D-d)}  {W^{(q)}(u-d)}  (\varphi_d (u;f)  - \varphi_d (U;f) )
\\ &+  Z^{(q)}(D-d) w(D) - \varphi_d (D;f).
\end{align*}
Solving this, we have
\begin{align}
w(D) &=  \frac {\Theta(d,D,U,u;f)} 
{q \Xi(d,D,U,u)}. \label{w_D}
\end{align}
Substituting this in \eqref{w_intermediate},
\begin{align*}
w(x) 
&= \frac {W^{(q)}(x-d)}  {W^{(q)}(u-d)}  \frac {\Theta(d,D,U,u;f)} 
{\Xi(d,D,U,u)}  \left[- \overline{W}^{(q)}(U-d,u-d)  -  \frac {W^{(q)}(U-d)}  {W^{(q)}(D-d)} \overline{W}^{(q)}(D-d) \right] \\ &+ \frac {W^{(q)}(x-d)}  {W^{(q)}(u-d)} \Big[ \varphi_d (u;f)  - \varphi_d (U;f)  + \frac {W^{(q)}(U-d)}  {W^{(q)}(D-d)} \varphi_d (D;f) \Big] \\
&+  Z^{(q)}(x-d) \frac {\Theta(d,D,U,u;f)} 
{q \Xi(d,D,U,u)} - \varphi_d (x;f). 
\end{align*}
In order to simplify this, note that
\begin{align*}
&  \frac {\Theta(d,D,U,u;f)} 
{\Xi(d,D,U,u)} \Big[ -\overline{W}^{(q)}(U-d,u-d) -  \frac {W^{(q)}(U-d)}  {W^{(q)}(D-d)} \overline{W}^{(q)}(D-d) \Big] \\
&=- \frac 1 {W^{(q)}(D-d)}  \frac {\Theta(d,D,U,u;f)} 
{\Xi(d,D,U,u)}  \left[ \Xi(d,D,U,u) + W^{(q)}(u-d) \overline{W}^{(q)} (D-d) \right] \\
&=-  (\varphi_d (u;f) - \varphi_d (U;f)  )  + \frac { 
{W^{(q)}(U-d,u-d)} \varphi_d (D;f) } 
{W^{(q)}(D-d)} \\
&- \frac {W^{(q)}(u-d) \overline{W}^{(q)} (D-d) } {W^{(q)}(D-d)}  \frac {\Theta(d,D,U,u;f)} 
{\Xi(d,D, U,u)}. 
\end{align*}
Substituting this,
\begin{align*}
w(x) &= -\frac {W^{(q)}(x-d)}  {W^{(q)}(u-d)}  { \Big[    \varphi_d (u;f) - \varphi_d (U;f)  
- \frac {W^{(q)}(U-d,u-d)} {W^{(q)}(D-d)}  \varphi_d (D;f)  \Big]} 
 \\   &- {W^{(q)}(x-d)}  \frac { \overline{W}^{(q)}(D-d)} {W^{(q)}(D-d)}  \frac { \Theta(d,D,U,u;f) } 
{ \Xi(d,D, U,u)}\\ &+ \frac {W^{(q)}(x-d)}  {W^{(q)}(u-d)} \Big[ \varphi_d (u;f)  - \varphi_d (U;f) + \frac {W^{(q)}(U-d)}  {W^{(q)}(D-d)} \varphi_d (D;f) \Big] \\
&+  Z^{(q)}(x-d) \frac {\Theta(d,D,U,u;f)} 
{q \Xi(d,D, U,u)} - \varphi_d (x;f) \\
&= \frac {W^{(q)}(x-d)} {W^{(q)}(D-d)}  \varphi_d (D;f)  - \varphi_d (x;f) \\ &+ \left[ \frac {Z^{(q)}(x-d)} q - {W^{(q)}(x-d)}  \frac { \overline{W}^{(q)}(D-d)} {W^{(q)}(D-d)}   \right] \frac {\Theta(d,D,U,u;f)} 
{ \Xi(d,D, U,u)},
\end{align*}
which completes the proof of (1).  The proofs for (2) and (3) are also immediate by the construction of $A$ and by \eqref{w_D}.
\end{proof}

\section{Concluding Remarks}  \label{section_conclusion}

We have studied the band policy with parameters $(d,D,U,u)$ and its associated NPV's of the controlling and holding costs.  We focused on the case that is driven by a general spectrally negative L\'{e}vy process.  Using the fluctuation theory, we expressed the NPV's using the scale function.
Here, we conclude this paper with its contributions as well as challenges in applying to solve the cash management problem where one wants to minimize the total NPV of the costs over the set of impulse controls.

In a cash management problem, an admissible policy is given by a set of nondecreasing processes $\pi := \{ R^\pi, L^\pi\}$ that are $\mathbb{F}$-adapted and increase only with jumps.  The objective is to minimize the sum of holding and controlling costs given by
\begin{multline*}
V^\pi(x) := \mathbb{E}_x  \Big[ \int_0^\infty e^{-qt} f(A^\pi_t) {\rm d} t  + \sum_{0 \leq t < \infty} e^{-qt} [ c_L (\Delta L_t^\pi + k_L) 1_{\{ \Delta L_t^\pi > 0\}}  + c_R (\Delta R_t^\pi  + k_R)1_{\{ \Delta R_t^\pi > 0\}} ] \Big], 
\end{multline*}
where $c_L, c_R \in \R$ and  $A_t^\pi := X_t + R^\pi_t - L^\pi_t$ is the resulting process controlled by the policy $\pi$.

It is clear that the band policies studied in this paper are admissible, and it is naturally conjectured that, under a certain (for instance, convexity) assumption on the holding cost function $f$, the optimal strategy is given by a band policy for a suitable choice of the parameters $(d,D,U,u)$.  

From the well-known existing results on impulse control, the candidate values of $(d,D,U,u)$ are first chosen so that  the value function becomes continuous/smooth at the levels $d$ and $u$, and its slopes at $D$ and $U$ equal, respective, the negative of the unit proportional cost for $R^\pi$ and the unit proportional cost for $L^\pi$.   More precisely, if $V^*$ is the value function, it is expected to satisfy the following:
\begin{align} \label{system_equations}
\begin{split}
V^{*'}(d-) &= V^{*'}(d+), \\
V^{*'}(D) &= -c_R, \\
V^{*'}(U) &= c_L, \\
V^{*'}(u-) &= V^{*'}(u+).
\end{split}
\end{align}
Here, for the case $X$ is of bounded variation, because of irregularity of the lower half-line (see, e.g., page 142 of \cite{Kyprianou_2006}), the first smooth fit condition is replaced with the continuous fit condition: $V^{*}(d-) = V^{*}(d+)$.

Using the analytical expressions of the NPV's under the band policy, these four equations can be written concisely in terms of the scale function.  In particular,  the asymptotic behaviors of the scale function near zero as summarized in Remark \ref{remark_smoothness_zero}(2) are expected to be helpful in simplifying these. In turn, the problem reduces to identifying the four parameters $(d,D,U,u)$ as a solution to the system of four equations.  Unfortunately, however, this is likely to become a big hurdle.  Because the equations turn out to be nonlinear and somewhat complicated, even the existence/uniqueness of a solution is expected to  be difficult to show.  With regard to this, we refer the reader to  \cite{baurdoux_yamazaki_2014, Bayraktar_2013, Leung_Yamazaki_2011, Hernandez_Yamazaki_2013, Loeffen_2009_2,  Yamazaki_2013}  for simpler cases where two (instead of four) parameters are sought.

After the four parameters  $(d,D,U,u)$ that satisfy \eqref{system_equations} are identified, the last step is to verify the optimality.  This is equivalent to showing that the candidate value function solves the QVI of \cite{Bensoussan_Lions_1984}.  This is indeed the most challenging part of the problem. However, there are several benefits about having the semi-explicit expressions  written in terms of the scale function.  First, the harmonicity on $(d,u)$ can be proven easily thanks to the smoothness of the scale function and because the processes $e^{-q (t \wedge \tau_0^- \wedge \tau_b^+)} W^{(q)} (X_{t \wedge \tau_0^- \wedge \tau_b^+})$,  $e^{-q  (t \wedge \tau_0^- \wedge \tau_b^+)} Z^{(q)} (X_{t \wedge \tau_0^- \wedge \tau_b^+})$,  $e^{-q  (t \wedge \tau_0^- \wedge \tau_b^+)} (\overline{Z}^{(q)} (X_{t \wedge \tau_0^- \wedge \tau_b^+}) + \mu/q)$, $t \geq 0$, for any fixed $b > 0$ are martingales.   In addition, the property given as Remark \ref{remark_smoothness_zero}(3) has been shown to be useful in the verification as in the existing results \cite{baurdoux_yamazaki_2014, Bayraktar_2013, Leung_Yamazaki_2011,  Hernandez_Yamazaki_2013, Yamazaki_2013}.

Overall, the cash management problem of this form is conjectured to be challenging to solve.  However, the results obtained in this paper would certainly be helpful and potentially lead to an efficient way of solving the problem.

\section*{Acknowledgements}
The author thanks the anonymous referee for constructive comments and suggestions.
K.\ Yamazaki is in part supported by MEXT KAKENHI grant number 26800092, the Inamori foundation research grant, and the Kansai University subsidy for supporting young scholars 2014.

\addtolength{\baselineskip}{-0.5\baselineskip}

\end{document}